\documentclass[12pt]{article}

\usepackage[mathscr]{eucal}
\usepackage{amsmath,amsthm,amssymb,graphicx}
\usepackage[colorlinks,linkcolor=blue,citecolor=red,urlcolor=blue]{hyperref}

\theoremstyle{remark}
\newtheorem*{df*}{\bf \indent Definition}
\newtheorem*{conj*}{\bf \indent Conjecture}
\newtheorem*{thm*}{\bf \indent Theorem}
\newtheorem*{lem*}{\bf \indent Lemma}
\newtheorem*{rem*}{\bf \indent Remark}
\newtheorem*{cor*}{\bf \indent Corollary}

\numberwithin{equation}{section}

\DeclareMathOperator{\lcm}{lcm}

\begin{document}

\begin{center}
\Large \textbf{On a Conjecture on Sharygin Triangles}
\end{center}

\begin{center}
\textbf{N.\,N.~Osipov}\\
Siberian Federal University\\
Krasnoyarsk, Russia
\end{center}

\begin{abstract}
By a simple method we prove the following conjecture on Sharygin triangles: there is only one Sharygin triangle (up to an isometry) whose vertices are chosen from the set of vertices of a regular polygon inscribed in a circle of radius $1$.

\end{abstract}

\tableofcontents

\section{Introduction and Preliminaries}
\label{sec-1}

\hspace{0.5cm} In the paper, we will use the same terminology as in \cite{Netay&Savvateev-2017}. Let us recall the most important definitions.

For a triangle $ABC$, let $A_1$, $B_1$, and $C_1$ be the intersection points of the \emph{internal bisectors} with opposite sides (here, $A_1$ lies on $BC$ etc.). We call the triangle $A_1B_1C_1$ the \emph{bisectral triangle} of a given triangle $ABC$.

\begin{df*}
A non-isosceles triangle $ABC$ is called a \emph{Sharygin triangle} if its bisectral triangle $A_1B_1C_1$ is isosceles.
\end{df*}

Surprisingly, but Sharygin triangles exist. Furthermore, each of them must have an obtuse angle lying in a very small interval between $\approx 102{.}663^\circ$ and $\approx 104{.}478^\circ$ (see for more details the solution of Problem 203 in \cite[Sec. 2]{Sharygin-1988}).

The most famous example of a Sharygin triangle $ABC$ is the so-called \emph{heptagonal triangle} with angles
\begin{equation}
\label{eq-1.1}
\angle A=\pi/7, \quad \angle B=2\pi/7, \quad \angle C=4\pi/7.
\end{equation}
For this triangle, we have
\begin{equation}
\label{eq-1.2}
C_1A_1=C_1B_1.
\end{equation}
The heptagonal triangle was firstly studied in \cite{Bankoff&Garfunkel-1973}. Later, it has been rediscovered several times (see, for instance, \cite{Tokarev-2006}).

It is easy to see that all the vertices of the heptagonal triangle can be placed in some vertices of a regular heptagon. Many different properties of the heptagonal triangle can be found in \cite[Ch. III]{Modenov-1979} (see Problems 11, 23, 45 in Sec. 1 and  14 in Sec. 2).

In \cite{Netay&Savvateev-2017}, the authors asked the following natural question: ``Are there other regular polygons such that three of its vertices form a Sharygin triangle?''. The computer experiments led them to the following conjecture (see Hypothesis 2.2 in \cite{Netay&Savvateev-2017}) which has been verified for all regular $N$-gons with $N \leqslant 2000$.

\begin{conj*}
Suppose that the vertices of a Sharygin triangle coincide with three vertices of a regular polygon. Then this triangle is similar to the heptagonal triangle.
\end{conj*}

The main aim of our paper is to prove it in the general case. In Section \ref{sec-2}, we prove Theorem from which we can very easily deduce the statement of Conjecture (see Corollary in Section \ref{sec-3}). We also find two triangles (the so-called pentadecagonal triangles) which have similar properties to the heptagonal triangle. In the Section \ref{sec-4}, we give some comments on the obtained results.

Denote by $\mathbb{N}$ the set of positive integers. For an arbitrary $N \in \mathbb{N}$, let
\[
\zeta_N=\exp{(2\pi\sqrt{-1}/N)}
\]
be the standard primitive $N$th root of unity in the field of complex numbers $\mathbb{C}$. Let
\[
f(x,y)=1+x+y+x^2y^3+x^3y^2+x^3y^3
\]
be the special symmetric polynomial in $x$, $y$ mentioned in \cite[Sec. 2]{Netay&Savvateev-2017}. According to the results obtained in \cite[Sec. 2]{Netay&Savvateev-2017}, we need to solve the system
\begin{equation}
\label{eq-1.3}
\text{$f(x,y)=0$ and $x^N=y^N=1$ for some $N \in \mathbb{N}$}
\end{equation}
over $\mathbb{C}$. Here, the unknowns complex numbers $x$ and $y$ have the following sense: their absolute values are $1$ while their arguments $\arg{x}$ and $\arg{y}$ coincide with two acute angles of a Sharygin triangle (recall that its third angle must be obtuse). Therefore, it is enough to prove that there is unique non-ordered pair $\{x,y\}$ satisfying \eqref{eq-1.3} for which real and imaginary parts of three numbers $x$, $y$, $xy$ are positive. Obviously, this pair $\{x,y\}$ must correspond to the heptagonal triangle that means $\{\arg{x},\arg{y}\}=\{\pi/7,2\pi/7\}$.

As it turns out later, a simplifying point is that the system \eqref{eq-1.3} has only two unknowns. Our idea is to construct a finite set of additional algebraic equations $h_i(x,y)=0$ such that any solution of \eqref{eq-1.3} must be a solution of one of the systems
\[
f(x,y)=0, \quad h_i(x,y)=0.
\]
Then we could just solve all such systems using a \emph{computer algebra system} (in view of \emph{Bezout's theorem} and absolute irreducibility of $f(x,y)$, we can a priori assume that each of them has only a finite set of solutions over $\mathbb{C}$).

\section{Main Result}
\label{sec-2}

\hspace{0.5cm} Let $\alpha=\zeta_3$, $\beta=\zeta_5$, $\gamma=\zeta_7$ and define the sets $S_0$, \dots, $S_3$ as follows:
\[
\begin{array}{l}
S_0=\{(1,-1),(-1,1)\},\\
S_1=\{(\alpha,\alpha),(\alpha^2,\alpha),(\alpha,\alpha^2),(\alpha^2,\alpha^2),\\
\hphantom{S_1=\{}(\alpha,-\alpha),(-\alpha,\alpha),(\alpha^2,-\alpha^2),(-\alpha^2,\alpha^2)\},\\
S_2=\{(\beta\alpha,\beta\alpha^2),(\beta^2\alpha,\beta^2\alpha^2),(\beta^3\alpha,\beta^3\alpha^2),(\beta^4\alpha,\beta^4\alpha^2),\\
\hphantom{S_2=\{}(\beta\alpha^2,\beta\alpha),(\beta^2\alpha^2,\beta^2\alpha),(\beta^3\alpha^2,\beta^3\alpha),(\beta^4\alpha^2,\beta^4\alpha)\},\\
S_3=\{(\gamma,-\gamma^4),(\gamma^2,-\gamma),(\gamma^3,-\gamma^5),(\gamma^4,-\gamma^2),(\gamma^5,-\gamma^6),(\gamma^6,-\gamma^3),\\
\hphantom{S_2=\{}(-\gamma,\gamma^2),(-\gamma^2,\gamma^4),(-\gamma^3,\gamma^6),(-\gamma^4,\gamma),(-\gamma^5,\gamma^3),(-\gamma^6,\gamma^5)\}.
\end{array}
\]
Also, let $S$ be the union of $S_0$, \dots, $S_3$ so that $|S|=30$.

\begin{thm*}
The set of all solutions $(x,y)$ of the system \eqref{eq-1.3} over $\mathbb{C}$ coincides with $S$.
\end{thm*}

Denote by $V$ the set $\{-1,-1/2,-1/4\}$. Let $\varphi$ and $\mu$ be the \emph{Euler totient function} and the \emph{M\"{o}bius function}, respectively (the standard properties of these functions can be found, for example, in \cite[\S A.5, \S A.6]{Nathanson-1996}). First we prove the following auxiliary result.

\begin{lem*}
For any $N \in \mathbb{N}$ and $k \in \mathbb{Z}$, if $M=N/(k,N)$ and $\mu(M)/\varphi(M) \in V$ then there is $t \in \mathbb{Z}$ such that $k=tN/2$ or $k=tN/3$ or $k=tN/5$. Here, $(k,N)$ denotes $\gcd{(k,N)}$.
\end{lem*}

\begin{proof}[\indent \textsc{Proof}]
We have to consider three cases with respect to the special values of $\mu(M)/\varphi(M)$ represented in $V$.

(a) If $\mu(M)/\varphi(M)=-1$ then $\mu(M)=-1$ and $\varphi(M)=1$. Therefore, $M=2$, $(k,N)=N/2$ and $k=tN/2$ for some $t \in \mathbb{Z}$.

(b) If $\mu(M)/\varphi(M)=-1/2$ then $\mu(M)=-1$ and $\varphi(M)=2$. This implies $M=3$ (indeed, if $\varphi(M)=2$ then $M \in \{3,4,6\}$ and only $M=3$ is suitable). Therefore, $(k,N)=N/3$ and we get $k=tN/3$ for some $t \in \mathbb{Z}$.

(c) If $\mu(M)/\varphi(M)=-1/4$ then $\mu(M)=-1$ and $\varphi(M)=4$. Here, we have $M=5$ (indeed, if $\varphi(M)=4$ then $M \in \{5,8,10,12\}$ and only $M=5$ is suitable). Therefore, $(k,N)=N/5$ and $k=tN/5$ fore some $t \in \mathbb{Z}$.

The proof is completed.
\end{proof}

We can now proceed to prove the main result of the paper.

\begin{proof}[\indent \textsc{Proof of Theorem}]
We can set $x=\zeta_N^a$, $y=\zeta_N^b$ for some integers $a$, $b$ and $N \in \mathbb{N}$. Then
\begin{equation}
\label{eq-2.1}
1+\zeta_N^a+\zeta_N^b+\zeta_N^{2a+3b}+\zeta_N^{3a+2b}+\zeta_N^{3a+3b}=0.
\end{equation}
Clearly, all summands in the left hand side \eqref{eq-2.1} are elements of the \emph{cyclotomic field} $\mathbb{Q}(\zeta_N)$. Let us recall some facts on cyclotomic fields (see, for instance, \cite[Ch. 2]{Washington-1997}). The \emph{automorphisms} of $\mathbb{Q}(\zeta_N)$ are exactly the maps defined by the rule $\zeta_N \mapsto \zeta_N^j$ where $j$ is a residue modulo $N$ with $\gcd{(j,N)}=1$. Denote by $R_N^*$ the set of all such residues $j$. Then we have
\[
[\mathbb{Q}(\zeta_N):\mathbb{Q}]=|R_N^*|=\varphi(N).
\]
This is a direct consequence of the fact that the so-called \emph{cyclotomic polynomial}
\[
\Phi_N(x)=\prod_{j \in R_N^*}(x-\zeta_N^j)
\]
is irreducible over $\mathbb{Q}$ (it has rational and even integer coefficients) and $\deg{\Phi_N(x)}=\varphi(N)$.

It follows from \eqref{eq-2.1} that
\[
1+\zeta_N^{aj}+\zeta_N^{bj}+\zeta_N^{(2a+3b)j}+\zeta_N^{(3a+2b)j}+\zeta_N^{(3a+3b)j}=0
\]
for each $j \in R_N^*$. After summation over $j \in R_N^*$, we obtain
\begin{equation}
\label{eq-2.2}
\varphi(N)+c_N(a)+c_N(b)+c_N(2a+3b)+c_N(3a+2b)+c_N(3a+3b)=0,
\end{equation}
where $c_N$ is the \emph{Ramanujan sum} defined by
\[
c_N(k)=\sum_{j \in R_N^*}\zeta_N^{kj}.
\]
It is a well known fact that
\begin{equation}
\label{eq-2.3}
c_N(k)=\frac{\varphi(N)\mu(N/(k,N))}{\varphi(N/(k,N))}
\end{equation}
(see, for example, \cite[\S A.7]{Nathanson-1996}).

For a convenient notation, let
\[
k_1=a, \quad k_2=b, \quad k_3=2a+3b, \quad k_4=3a+2b, \quad k_5=3a+3b
\]
so that \eqref{eq-2.1} and \eqref{eq-2.2} become
\[
1+\sum_{i=1}^5\zeta_N^{k_i}=0, \quad \varphi(N)+\sum_{i=1}^5c_N(k_i)=0,
\]
respectively. Applying \eqref{eq-2.3} and reducing by $\varphi(N)$, rewrite the last equality as
\begin{equation}
\label{eq-2.4}
1+\sum_{i=1}^5\frac{\mu(M_i)}{\varphi(M_i)}=0
\end{equation}
where $M_i=N/(k_i,N)$ for $i=1,\dots,5$.

Note that, for every $M \in \mathbb{N}$, a fraction of the shape $\mu(M)/\varphi(M)$ must be equal to $0$, $\pm 1$, or $\pm 1/m$ where $m \in \mathbb{N}$ is even. Based on Lemma, one can deduce the following: if one of the fractions $\mu(M_i)/\varphi(M_i)$ belongs to the set $V$ then the corresponding $k_i$ can be represented as $k_i=tN/2$ or $k_i=tN/3$ or $k_i=tN/5$ with some integer $t$. It is important to remark that the condition
\[
\text{``one of the fractions $\mu(M_i)/\varphi(M_i)$ is an element of $V$''}
\]
must be fulfilled. Indeed, otherwise there are at most five negative fractions $\mu(M_i)/\varphi(M_i)$ and each of them in absolute value does not exceed $1/6$; in this case, the equality \eqref{eq-2.4} cannot be true.

Let us consider all indicated cases for $k_i$'s. Introduce the monomials
\[
g_1(x,y)=x, \quad g_2(x,y)=y, \quad g_3(x,y)=x^2y^3, \quad g_4(x,y)=x^3y^2, \quad g_5(x,y)=x^3y^3
\]
and note that for $x=\zeta_N^a$ and $y=\zeta_N^b$ the value of $g_i(x,y)$ is $\zeta_N^{k_i}$ ($i=1,\dots,5$). If $k_i=tN/2$ then $\zeta_N^{2k_i}=1$. This means that we obtain an additional equation $g_i(x,y)^2=1$ for some $i$. Likewise, in two remaining cases $k_i=tN/3$ and $k_i=tN/5$ we get an additional equation $g_i(x,y)^3=1$ and $g_i(x,y)^5=1$ with some $i$, respectively.

Thus, it just remains to solve over $\mathbb{C}$ fifteen systems of algebraic equations of the form
\begin{equation}
\label{eq-2.5}
f(x,y)=0, \quad g_i(x,y)^l=1
\end{equation}
where $i=1,\dots,5$ and $l \in \{2,3,5\}$. More precisely, we have to determine those solutions $(x,y)$ for which there is $N \in \mathbb{N}$ such that $x^N=y^N=1$ (let us call such $(x,y)$ the \emph{special solutions}). A priori, $N$ may depend on $(x,y)$ but below we show that one can take $N=210$ for any special solution $(x,y)$.

For this purpose, we can proceed different well-known techniques. We prefer \emph{Gr\"{o}bner bases}, but we could just as well use \emph{resultants} (regarding these notions, we refer the reader to any book on algebra or computer algebra,  for example \cite{Gathen&Gerhard-2013}). For computing with polynomials, we plan use a computer algebra system (for instance, Maple CAS \cite{maple}).

For each of the systems \eqref{eq-2.5}, we compute Gr\"{o}bner basis $G$ with respect to \emph{pure lexicographic order} $y \succ x$. The first polynomial in $G$ depends only on $x$ and has rational coefficients since both $f(x,y)$ and $g_i(x,y)$ are in $\mathbb{Q}[x,y]$. Denote this polynomial by $F(x)$. Further, we factorize $F(x)$ over $\mathbb{Q}$ and find out whether one of its irreducible factors $p(x)$ coincides with some cyclotomic polynomial $\Phi_n(x)$.

Eventually, after factorization of all such polynomials $F(x)$ for all systems \eqref{eq-2.5}, we obtain a finite set $K$ of pairwise distinct monic irreducible factors $p(x)$ so that if $p(x) \in K$ then
\[
d=\deg{p(x)} \in D=\{1,2,4,6,8,12,16,28\}.
\]
As an example, represent an irreducible factor of degree $d=12$ received for the system \eqref{eq-2.5} with $i=2$ and $l=5$:
\[
p_{12}(x)=x^{12}+2x^{11}+6x^{10}+5x^9-8x^7-11x^6-8x^5+5x^3+6x^2+2x+1.
\]
If $p_{12}(x)=\Phi_n(x)$ for some $n$ then $\deg{\Phi_n(x)}=\varphi(n)=12$ and $n \in \{13,21,26,28,36,42\}$. But, for such $n$, we actually have $p_{12}(x) \neq \Phi_n(x)$.

In general, for all $p(x) \in K$, a possible coincidence $p(x)=\Phi_n(x)$ for some $n$ implies
\begin{gather*}
n \in I_1=\\
\{\text{$1$ through $10$},12,13,14,15,16,17,18,20,21,24,26,28,29,30,32,34,36,40,42,48,58,60\}.
\end{gather*}
Let $L=\{\Phi_n(x):n \in I_1\}$. In fact, we obtain $K \cap L=\{\Phi_n(x):n \in I_2\}$ where
\[
I_2=\{1,2,3,5,6,7,14,15\} \subset I_1.
\]
Since $\lcm{(I_2)}=210$, we arrive at $x^{210}=1$ for any special solution $(x,y)$. We also have $y^{210}=1$ due to a symmetry.

We can now determine all solutions $(x,y)=(\zeta_N^a,\zeta_N^b)$ of the system \eqref{eq-1.3} just letting $N=210$ and using brute force search for $(a,b)$. As a result, we obtain the following set $P$ of pairs $(a,b)$: two pairs
\[
(70,70),\;(140,140)
\]
with $a=b$ and twenty eight pairs with $a \neq b$ which are
\[
\begin{array}{l}
(0,105),\;(14,154),\;(15,30),\;(28,98),\;(35,140),\;(45,90),\;(56,196),\;(60,135),\;(70,140),\\
(70,175),\;(75,150),\;(112,182),\;(120,165),\;(180,195)
\end{array}
\]
together with symmetric ones. Thus, the set of all solutions of the system \eqref{eq-1.3} is
\begin{equation}
\label{eq-2.6}
\{(\zeta_{210}^a,\zeta_{210}^b):(a,b) \in P\}.
\end{equation}
One can verify that \eqref{eq-2.6} is exactly the set $S$ defining before. This completes the proof.
\end{proof}

\begin{rem*}
Surely, in the final part of the proof, we could solve only five systems \eqref{eq-2.5} with a single value $l=\lcm{(2,3,5)}=30$. However, this would make the corresponding sets $D$, $I_1$, $I_2$ longer. Moreover, it would be more difficult to compute Gr\"{o}bner base $G$ for those monomials $g_i(x,y)$ for which $\deg{(g_i(x,y)^{30}-1)}$ admits large values (in this case, we would have to compute some resultants). Anyway, it's clear that this part of the proof is purely technical.
\end{rem*}

\section{Proof of a Conjecture}
\label{sec-3}

\hspace{0.5cm} Based on Theorem, we can now prove Conjecture.

\begin{cor*}
The statement of Conjecture is true.
\end{cor*}

\begin{proof}[\indent \textsc{Proof}]
We have only to highlight those solutions $(x,y) \in S$ for which real and imaginary parts of three numbers $x$, $y$, and $xy$ are positive. A direct checking shows us that this condition is satisfied for only one (up to a symmetry) solution
\begin{equation}
\label{eq-3.1}
(x,y)=(\gamma,-\gamma^4)
\end{equation}
which lies in $S_3$. Here, we have $\arg{x}=2\pi/7$ and $\arg{y}=\pi/7$. As noted before, this leads us to the heptagonal triangle.
\end{proof}

Meanwhile, the following question remains: do the other solutions $(x,y)$ from $S$ make any geometrical sense? It's clear that only the solutions $(x,y) \in S_2$ can be of interest.

Below, we show how the solution
\begin{equation}
\label{eq-3.2}
(x,y)=(\beta\alpha,\beta\alpha^2)
\end{equation}
might be interpreted geometrically. For this purpose, we interpret the Euclidean plan as the \emph{complex plan} $\mathbb{C}$. Let us define a triangle $ABC$ by the tangency points
\[
T_a=z_1, \quad T_b=z_2, \quad T_3=z_3
\]
of the \emph{inscribed circle} with the sides $BC$, $CA$, $AB$, respectively. Here, we assume the parameters $z_i$'s to be complex numbers with absolute values $1$ so that
\[
A=\frac{2z_2z_3}{z_2+z_3}, \quad B=\frac{2z_1z_3}{z_1+z_3}, \quad C=\frac{2z_1z_2}{z_1+z_2}
\]
(see \cite{Osipov-2014} for more details about such a parametrization of triangles on Euclidean plan). Due to geometrical reasons, the triangle $T_aT_bT_c$ must be acute-angled. Then we have
\[
A_1=\frac{2z_1z_2z_3}{z_1^2+z_2z_3}, \quad B_1=\frac{2z_1z_2z_3}{z_2^2+z_3z_1}, \quad C_1=\frac{2z_1z_2z_3}{z_3^2+z_1z_2}.
\]
One can verify that if the triangle $ABC$ is not isosceles then the condition \eqref{eq-1.2} is equivalent to the following equation:
\[
z_1^3z_2+z_2^3z_1-z_1^3z_3-z_2^3z_3+z_1^2z_3^2+z_2^2z_3^2=0.
\]
Letting here $z_3=-1$, $z_1=x$, $z_2=y^{-1}$, we get exactly our equation
\[
f(x,y)=0.
\]

For the solution \eqref{eq-3.1}, three points
\[
T_a=x, \quad T_b=y^{-1}, \quad T_c=-1
\]
actually form an acute-angled triangle. Moreover, a direct checking shows that the corresponding non-isosceles triangle $ABC$ whose vertices are
\[
A=-\frac{2}{1-y}, \quad B=\frac{2x}{1-x}, \quad C=\frac{2x}{1+xy}
\]
is exactly the heptagonal triangle with angles \eqref{eq-1.1}.

On the contrary, for the solution \eqref{eq-3.2}, the triangle $T_aT_bT_c$ is not acute-angled. Nevertheless, the corresponding triangle $ABC$ exists and its angles are
\begin{equation}
\label{eq-3.3}
\angle A=11\pi/15, \quad \angle B=\pi/15, \quad \angle C=\pi/5.
\end{equation}
We need to interpret correctly (from a geometrical point of view) the equality \eqref{eq-1.2}.

For an arbitrary non-isosceles triangle $ABC$, let $A_2$, $B_2$, and $C_2$ be the intersection points of the \emph{external bisectors} with opposite sides. For the triangle $ABC$ generated by the solution \eqref{eq-3.2}, the (formally) internal bisectors $AA_1$ and $BB_1$ are actually its external bisectors $AA_2$ and $BB_2$, respectively. Thus, for the triangle $ABC$ with angles \eqref{eq-3.3}, the equality \eqref{eq-1.2} must be interpreted as
\[
C_1A_2=C_1B_2.
\]
This yields a desired geometrical interpretation of \eqref{eq-3.2}.

Similarly, the solution
\begin{equation}
\label{eq-3.4}
(x,y)=(\beta^2\alpha,\beta^2\alpha^2)
\end{equation}
brings us another non-isosceles triangle $ABC$ with angles
\begin{equation}
\label{eq-3.5}
\angle A=2\pi/15, \quad \angle B=7\pi/15, \quad \angle C=6\pi/15
\end{equation}
but now the correct way to interpret \eqref{eq-1.2} is
\[
C_2A_1=C_2B_2.
\]

For a triangle $ABC$, the corresponding triangles $A_1B_2C_2$, $A_2B_1C_2$, and $A_2B_2C_1$ also will be called the bisectral triangles. As the vertices of the triangles with angles \eqref{eq-3.3} and \eqref{eq-3.5} can be placed in suitable vertices of a regular pentadecagon, we call them the \emph{pentadecagonal triangles} (the \emph{first} and the \emph{second}, respectively).

\section{Concluding Remarks}
\label{sec-4}

\hspace{0.5cm} Let us comment on our results and discuss some other ways of proving Theorem.

\textbf{1.} How diverse can equalities \eqref{eq-2.4} be if we go through all thirty solutions \eqref{eq-2.6}? A direct computation shows that the complete list of such equalities is the following:
\[
\begin{array}{l}
1+1-1-1+1-1=0,\vspace{2mm}\\
1-1+1+1-1-1=0
\end{array}
\]
(without fractions, only $\pm 1$) and
\[
\begin{array}{l}
1+\dfrac{1}{2}-\dfrac{1}{2}-\dfrac{1}{2}+\dfrac{1}{2}-1=0,\vspace{2mm}\\
1-\dfrac{1}{2}+\dfrac{1}{2}+\dfrac{1}{2}-\dfrac{1}{2}-1=0,\vspace{2mm}\\
1+\dfrac{1}{6}-\dfrac{1}{6}-\dfrac{1}{6}-1+\dfrac{1}{6}=0,\vspace{2mm}\\ 1-\dfrac{1}{6}+\dfrac{1}{6}-1-\dfrac{1}{6}+\dfrac{1}{6}=0,\vspace{2mm}\\
1-\dfrac{1}{2}-\dfrac{1}{2}-\dfrac{1}{2}-\dfrac{1}{2}+1=0,\vspace{2mm}\\ 1+\dfrac{1}{8}+\dfrac{1}{8}-\dfrac{1}{2}-\dfrac{1}{2}-\dfrac{1}{4}=0.
\end{array}
\]
One can see that all numbers from the set $V=\{-1,-1/2,-1/4\}$ are actually represented here. Meanwhile, there are other equalities which could to be \eqref{eq-2.4}, namely
\begin{equation}
\label{eq-4.1}
\begin{array}{l}
1+0-\dfrac{1}{4}-\dfrac{1}{4}-\dfrac{1}{4}-\dfrac{1}{4}=0,\vspace{2mm}\\
1-\dfrac{1}{4}-\dfrac{1}{4}-\dfrac{1}{4}-\dfrac{1}{6}-\dfrac{1}{12}=0,\vspace{2mm}\\
1-\dfrac{1}{4}-\dfrac{1}{4}-\dfrac{1}{4}-\dfrac{1}{8}-\dfrac{1}{8}=0,\vspace{2mm}\\
1-\dfrac{1}{4}-\dfrac{1}{4}-\dfrac{1}{6}-\dfrac{1}{6}-\dfrac{1}{6}=0.
\end{array}
\end{equation}
Since the system \eqref{eq-1.3} has only two unknowns $x$ and $y$ (and, consequently, there is a simpler way which was chosen by us), we do not need to study all of them.

\textbf{2.} Among the components of all pairs $(x,y) \in S$ there are all primitive $n$th roots of unity whose order $n \in I_2$ except $n=5$. For instance, we have
\[
-\alpha^2=\zeta_6, \quad -\gamma^4=\zeta_{14}, \quad \beta^2\alpha^2=\zeta_{15}.
\]
In order to clarify $I_2$, we can modify the proof of Theorem as follows. Let $V'=\{-1,-1/2\}$. If the condition
\begin{equation}
\label{eq-4.2}
\text{``one of the fractions $\mu(M_i)/\varphi(M_i)$ is an element of $V'$''}
\end{equation}
is satisfied then we solve ten systems \eqref{eq-2.5} with $i=1,\dots,5$ and $l \in \{2,3\}$ as described above. In this case, the corresponding sets $D$, $I_1$, $I_2$ will be shorter, namely
\[
\begin{array}{l}
D=\{1,2,6,8,12\},\\
I_1=\{1,2,3,4,6,7,9,13,14,15,16,18,20,21,24,26,28,30,36,42\},\\
I_2=\{1,2,3,6,7,14,15\}.
\end{array}
\]
Solving all such systems, we arrive at the same set $S$ of solutions $(x,y)$ as earlier. But we have to consider the non-empty case when \eqref{eq-4.2} is not satisfied. In this case, the equality \eqref{eq-2.4} must be one of the equalities \eqref{eq-4.1}. Next, we can, for example, solve twenty systems of the form
\[
f(x,y)=0, \quad g_{i_1}(x,y)^l=g_{i_2}(x,y)^l=g_{i_3}(x,y)^l=1
\]
with $1 \leqslant i_1<i_2<i_3 \leqslant 5$ and $l \in \{5,7\}$ (see Lemma which needs to be supplemented with the following: if $\mu(M)/\varphi(M)=-1/6$ then $k=tN/7$ where $t \in \mathbb{Z}$). Finally, it occurs that all these systems have no solutions $(x,y)$.

\textbf{3.} Any pair $(x,y) \in S$ yields a certain identity of the form
\begin{equation}
\label{eq-4.3}
1+\sum_{i=1}^5\xi_i=0
\end{equation}
where $\xi_1=x$, $\xi_2=y$, $\xi_3=x^2y^3$, $\xi_4=x^3y^2$, $\xi_5=x^3y^3$ are some roots of unity. Thus, we have some \emph{vanishing sums} of roots of unity which was studied in \cite{Conway&Jones-1976} systematically. In our case, the equalities \eqref{eq-4.3} can be of the following form:
\[
(1-1)+(1-1)+(1-1)=0
\]
in the case $(x,y) \in S_0$,
\[
(1-1)+(\alpha-\alpha)+(\alpha^2-\alpha^2)=0, \quad
(1+\alpha+\alpha^2)+(1+\alpha+\alpha^2)=0
\]
in the case $(x,y) \in S_1$,
\[
(1+\alpha+\alpha^2)+\beta^l(1+\alpha+\alpha^2)=0 \quad (l=1,\dots,4)
\]
in the case $(x,y) \in S_2$, and
\[
(1-1)+(\gamma^l-\gamma^l)+(\gamma^{\sigma(l)}-\gamma^{\sigma(l)})=0 \quad (l=1,\dots,6)
\]
in the case $(x,y) \in S_3$, where $\sigma$ is the permutation $(142)(356)$. This corresponds to Theorem 6 from \cite{Conway&Jones-1976} which provides a description of all \emph{non-empty} vanishing sums of roots of unity of length at most nine. Additionally, Theorem 6 suggests an alternative way to prove our Theorem, but the proof of Theorem 6 itself is quite difficult.

\textbf{4.} The method proposed in the proof of Theorem is more elementary and simpler compared to one from \cite{Conway&Jones-1976}. It can be applied to other \emph{trigonometric diophantine equations} as the \emph{Gordan equation} \cite{Conway&Jones-1976,Gordan-1877}. This equation in our notation can be written as
\[
2+x+x^{-1}+y+y^{-1}+z+z^{-1}=0.
\]
Here, we have three unknowns $x$, $y$, $z$ satisfying the additional condition
\[
\text{$x^N=y^N=z^N=1$ for some $N \in \mathbb{N}$}.
\]
Also, the method works for a similar equation
\[
1+x+x^{-1}+y+y^{-1}+z+z^{-1}=0
\]
(see \cite{Osipov&Pazii-2024}) and for the \emph{Coxeter} (or \emph{Crosby}) \emph{equation} \cite{Conway&Jones-1976,Crosby-1946}
\[
x+x^{-1}+y+y^{-1}+z+z^{-1}=0
\]
which must be rewritten previously as
\[
1+x^2+xy+xy^{-1}+xz+xz^{-1}=0
\]

\textbf{5.} As it turns out, the second pentadecagonal triangle is \emph{algebraically conjugate} to the first one that means the following. Letting $\delta=\zeta_{15}$, rewrite the solutions \eqref{eq-3.2} and \eqref{eq-3.4} as $(\delta^8,\delta^{13})$ and $(\delta^{11},\delta)$, respectively. Then the automorphism of the field $\mathbb{Q}(\delta)$ defined by $\delta \mapsto \delta^7$ sends the first pair $(\delta^8,\delta^{13})$ to the second pair $(\delta^{11},\delta)$. However, this idea of ``algebraic replication'' gives us at most two triangles that are geometrically distinct.

What happens if we apply such an ``algebraic replication'' to the heptagonal triangle $ABC$ with angles \eqref{eq-1.1}? Expectedly, no new triangle will be discovered, but a new isosceles bisectral triangle for $ABC$ will be detected, namely the triangle $A_2B_2C_1$ for which the equality
\[
A_2B_2=A_2C_1
\]
holds (in \cite{Bankoff&Garfunkel-1973} and \cite[Ch. III]{Modenov-1979}, this fact seems not to be noted).

\textbf{6.} For the first pentadecagonal triangle, we also have
\begin{equation}
\label{eq-4.4}
AA_2=BB_2.
\end{equation}
This gives us an example of a non-isosceles triangle with two equal external bisectors.\footnote{For comparison: according to \emph{Steiner's theorem}, there is no any non-isosceles triangle with two equal internal bisectors.} But for the second pentadecagonal triangle, we have only
\begin{equation}
\label{eq-4.5}
AA_1=BB_2.
\end{equation}
One can prove that in the class of triangles $ABC$ whose angles are commensurable with $\pi$, the equality \eqref{eq-4.4} (\eqref{eq-4.5}, respectively) holds only for the first (second, respectively) pentadecagonal triangle. The proof is based on the Gordan equation.

\begin{figure}
\centering
\includegraphics{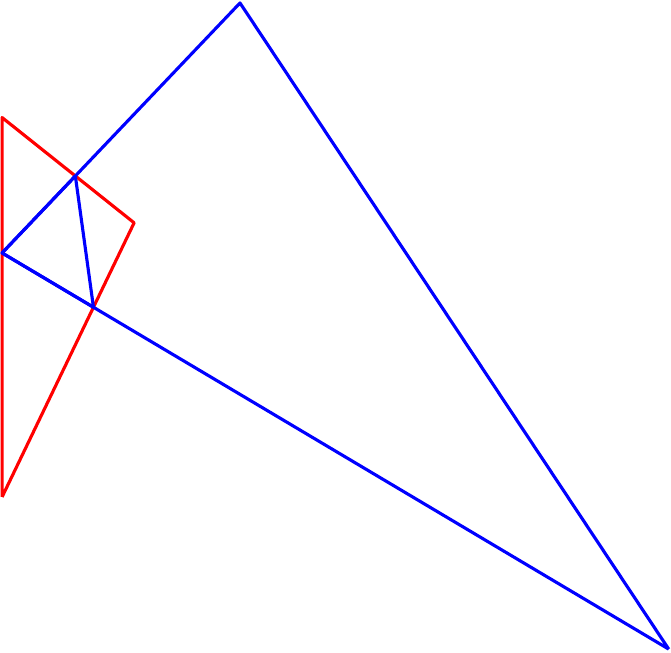}
\caption{The heptagonal triangle and its two isosceles bisectral triangles.}
\end{figure}

\begin{figure}
\centering
\includegraphics{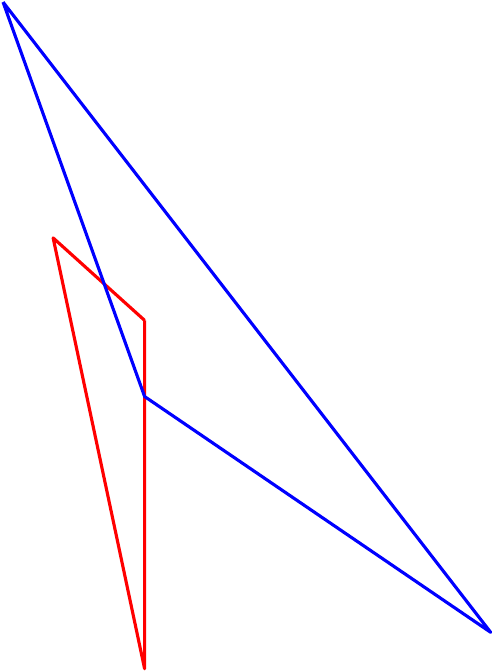}
\caption{The first pentadecagonal triangle and its isosceles bisectral triangle.}
\end{figure}

\begin{figure}
\centering
\includegraphics{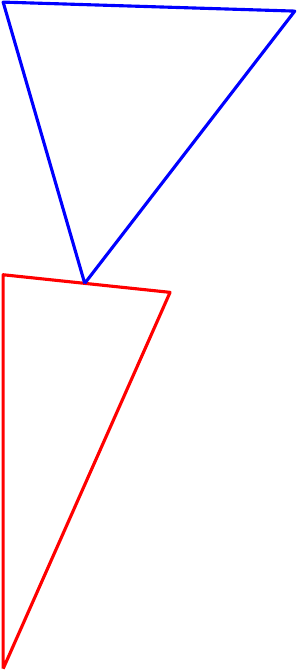}
\caption{The second pentadecagonal triangle and its isosceles bisectral triangle.}
\end{figure}

\end{document}